\documentclass[11pt]{amsart}

\usepackage{hyperref}

\usepackage{graphicx, enumerate, url}
\usepackage{amssymb,mathtools}
\usepackage{algorithm}
\usepackage{algpseudocode}
\usepackage{color}
\usepackage{tikz}
\numberwithin{equation}{section}
\numberwithin{algorithm}{section}

\theoremstyle{plain}
\newtheorem{theorem}{Theorem}[section]

\newtheorem{lemma}[theorem]{Lemma}
\newtheorem{corollary}[theorem]{Corollary}

\newtheorem{definition}[theorem]{Definition}

\theoremstyle{remark}

\DeclareMathOperator{\tr}{Tr}
\DeclareMathOperator{\diag}{diag}

\newcommand\defeq{\stackrel{\mathclap{\normalfont\mbox{def}}}{=}}
\newcommand{\N}{\mathbb{N}}

\newcommand{\R}{\mathbb{R}}

\newcommand{\x}{\mathbf{x}}

\newcommand{\y}{\mathbf{y}}

\newcommand{\U}{\mathbf{U}}

\newcommand{\cC}{\mathcal{C}}

\newcommand{\cJ}{\mathcal{J}}

\newcommand{\cU}{\mathcal{U}}
\newcommand{\cV}{\mathcal{V}}
\newcommand{\cW}{\mathcal{W}}
\newcommand{\cX}{\mathcal{X}}

\newcommand{\be}{\mathbf{e}}
\newcommand{\bbf}{\mathbf{f}}
\newcommand{\bg}{\mathbf{g}}
\newcommand{\bh}{\mathbf{h}}
\newcommand{\bp}{\mathbf{p}}
\newcommand{\bq}{\mathbf{q}}

\newcommand{\bu}{\mathbf{u}}
\newcommand{\bv}{\mathbf{v}}
\newcommand{\bw}{\mathbf{w}}
\newcommand{\bx}{\mathbf{x}}
\newcommand{\by}{\mathbf{y}}

\newcommand{\rB}{\mathrm{B}}
\newcommand{\rC}{\mathrm{C}}
\newcommand{\rD}{\mathrm{D}}

\newcommand{\rL}{\mathrm{L}}
\newcommand{\rS}{\mathrm{S}}
\newcommand{\rU}{\mathrm{U}}
\newcommand{\rV}{\mathrm{V}}

\newcommand{\0}{\mathbf{0}}
\newcommand{\1}{\mathbf{1}}

\begin{document}
\title[IPM in tensor optimal transport]{Interior point method in tensor optimal transport}
\author[Shmuel Friedland]{Shmuel~Friedland}
\address{Department of Mathematics and Computer Science, University of Illinois at Chicago, Chicago, Illinois, 60607-7045, USA }
\email{friedlan@uic.edu}
\date{October 29, 2023}
\begin{abstract}  
We study a tensor optimal transport (TOT) problem for $d\ge 2$ discrete  measures. This is a linear programming problem on $d$-tensors.     We introduces an interior point method (ipm) for $d$-TOT with a corresponding barrier function.   Using  a "short-step" ipm following central path within $\varepsilon$ precision we estimate the number of iterations. 
\end{abstract}


\maketitle
\section{Introduction} \label{sec:intro}
For $i\in\{1,2\}$ let $X_i$ be a random variables on $\Omega_i$ which has a finite number of values: $X_i:\Omega_i\to [n_i]$, where $[n]\defeq\{1,\ldots,n\}\subset \mathbb{N}$.  
 Assume that $\bp_i=(p_{1,i},\ldots,p_{n_i,i})$ is the column probability vector that gives the distribution of $X_i$: $\mathbb{P}(X_i=j)=p_{j,i}$.
Then the discrete Kantorovich optimal transport problem (OT) can be states as follows \cite{Kan}.  (See \cite{Vil03,Vil09} for modern account of OT.)  Let $Z$ be a random variable $Z: \Omega_1\times \Omega_2\to [n_1]\times[n_2]$ with contingency matrix (table)  $U\in\R_+^{n_1\times n_2}$ that gives the distribution of $Z$:
$\mathbb{P}(Z=(j_1,j_2))=u_{j_1,j_2}$.  
(Here $\R_+=[0,\infty), \R_{++}=(0,\infty)$.)
Let $P=(\bp_1,\bp_2)$ and  $\rU(P)$ be the convex set of all probability matrices with marginals $\bp_1,\bp_2$:
\begin{equation}\label{defUPmat}
\rU(P)=\{U=[u_{j_1,j_2}]\in\R_+^{n_1\times n_2}, \sum_{j_2=1}^{n_2} u_{j_1,j_2}=p_{j_1,1}, \quad \sum_{j_1=1}^{n_1} u_{j_1,j_2}=p_{j_2,2}\}.
\end{equation}
Let $C=[c_{i_1,i_2}]\in \R^{n_1\times n_2}$ be the cost matrix of transporting a unit $j_1\in[n_1]$ to $j_2\in[n_2]$.  Then the optimal transport problem is the linear programming problem (LP): 
\begin{equation}\label{TOTmat}
\tau(C,P)=\min\{\langle C,U\rangle, U\in \rU(P)\}. 
\end{equation}
(Here $\langle C,U\rangle=\tr C^\top U$.) 
For $n_1=n_2=n$ and a nonnegative symmetric cost matrix $C=[c_{ij}]$ with zero diagonal satisfying the triangle inequality $c_{ik}\le c_{ij}+c_{jk}$, the quantitty 
 $\tau(C,P)$ gives rise to a distance between probability vectors $\bp_1$ and $\bp_2$, which can be viewed as two histograms. 
It turns out that  $\tau(C,P)$ has many recent applications in machine learning \cite{AWR17,ACB17,LG15,MJ15,SL11}, statistics  \cite{BGKL17,FCCR18,PZ16,SR04} and computer vision \cite{BPPH11,RTG00}.   

A related problem to OT is quantum optimal transport (QOT), see \cite{FECZ22,CEFZ23} and references therein.    QOT is a semidefinite programming problem  \cite{VB96}, which are effectively solved using the interior point methods (ipm) \cite{NN94,Ren01,YFFKNN}.  In this paper we don't treat QOT, but we do use ipm
for solving OT and d-multi-marginal transport problem  that we call $d$-tensor optimal transport abbreviated as $d$-TOT.

Assume that $n_1=n_2=n$.  Then the complexity of finding $\tau(C,P)$ is $O(n^3 \log n)$, as this problem can be stated in terms of flows \cite{PW09}.  In applications, when $n$ exceeds a few hundreds,
the cost is prohibitive.  One way to improve the computation of $\tau(C,P)$ is to replace the linear programming with problem of OT with convex optimization by introducing an entropic regularization term as in \cite{Cut13}.  This regularization terms gives an $\varepsilon$-approximation to $\tau(C,P)$, where $\varepsilon>0$ is given.  The regularization term gives almost linear time approximation $O(n^2)$, ignoring the logarithmic terms, using a variation of the celebrated Sinkhorn algorithm for matrix diagonal scaling \cite{AWR17,LHCJ22,Fri20}.   

The aim of this paper is to introduce the interior point method for $d$-TOT,
 which correspond to the set of $d$-probability measures $\bp_i\in\R^{n_i}$ for $i\in[d]$.   
For the case $d=2$ the set $\rU(P)$ is the set of probability matrices $U\in \R_+^{n_1\times n_2}$ satisfying the marginal conditions $U \1_{n_2}=\bp_1, U^\top \1_{n_1}=\bp_2$.  Here $\1_n\in\R^n$ is the vector whose coordinates are all $1$.
For $d\ge 3$ we introduce $d$-mode tensors $\otimes_{k=1} ^d \R^{n_k}$.  We denote by $\cU\in  \otimes _{k=1}^d \R^{n_k}$ a tensor whose entries are $u_{i_1,\ldots,i_d}$, i.e., $\cU=[u_{i_1,\ldots,i_d}]$.  Assume that $\cC,\cU \in  \otimes _{k=1}^d \R^{n_k}, \cX\in \otimes^{j\in[d]\setminus\{k\}}\R^{n_j}$.   Denote by $\langle \cC,\cU\rangle$ the Hilbert-Schmidt inner product $\sum_{i_k\in[n_k],k\in[d]} c_{i_1,\ldots,i_d}u_{i_1,\ldots,i_d}$.  For $k\in [d]$ denote by  $\y=\cU\times_{\bar k} \cX\in \R^{n_k}$ the contraction on all but the index $k$:
$$y_{i_k}=\sum_{i_j\in[n_j], j\in[d]\setminus\{k\}}u_{i_1,\ldots,i_d} x_{i_1,\ldots,i_{k-1},i_{k+1},\ldots,i_d}.$$
Let $\cJ_{d-1,k}\in\otimes_{j\in [d]\setminus\{k\}}\R^{n_j}$ be the tensor whose all coordinates are $1$.
Define
\begin{equation}\label{defU(P)ten}
\rU(P)=\{\cU\in \otimes_{k=1}^d\R_+^{n_k}, \cU\times_{\bar k} \cJ_{d-1,k}=\bp_k, k\in[d]\},  P=(\bp_1,\ldots,\bp_d).
\end{equation}
Then the tensor optimal transport problem (TOT) is
\begin{eqnarray}\label{TOT}
\tau(\cC,P)\defeq \min\{\langle \cC, \cU\rangle, \cU\in \rU(P)\}.
\end{eqnarray}
TOT problem is a LP problem with $\prod_{k=1}^d n_k$ nonnegative variables and $1+\sum_{k=1}^d (n_k-1)$ constraints. 
The TOT was considered in \cite{Pi68,Po94} in the context of multidimensional assignment problem, where the entires of the tensor $\cU$ are either $0$ or $1$.
There is a vast literature on continuous multidimensional optimal transport problem.
See  for example \cite{FV18,BCN19,TDGU20,HRCK,LHCJ22} and the references therein. 
The TOT problem can be viewed as a discretization the  continuous multidimensional optimal transport problem.

We point out that that the ipm approach is easily adopted for variations of TOT.  Indeed,  for $d>2$ there is another well known variation of the set marginals.  Namely,  let $Z$ be a random variable $Z: \Omega_1\times\cdots \times\Omega_d\to [n_1]\times\cdots\times [n_d]$ with contingency tensor (table)  $\cU\in \R_+^{n_1\times\cdots n_d}$ that gives the distribution of $Z$:
$\mathbb{P}\big(Z=(j_1,\ldots, j_d)=u_{j_1,\ldots,j_d}\big)$.   Denote by $P_k\in \R_+^{n_1\times \cdots n_{k-1}\times n_{k+1}\times \cdots\times n_d}$ the marginal of $Z_k$ obtained from $Z$ with respect to $X_k:\Omega_k\to [n_k]$.    Assume that we are given 
 a distribution of $Z_0$ with a positive contingency tensor $\cV=(v_{j_1,\ldots,j_d})\in \otimes_{i=1}^d\R^{n_i}$ that have the above marginals.  Let $\rU(P)$ be a nonempty set of such $Z$.  See for example \cite{LZZ17} and references therein. Then one can use a similar ipm algorithm to find an approximate algorithm to the problem \eqref{TOT}.
 
 In subsection \ref{subsec:cest} we show that if $\bp_1,\ldots,\bp_d\in\R_+^n$ are weak $K-\ell$ uniform distributions, see Definition \ref{defwKldist}, then the number iterations of the ipm algorithm has complexity $O(n^{d/2})$, ignoring the logarithmic terms.  

\section{The interior point method}\label{sec:ipm}
We first recall some notations and definitions that we will use in this section.
\begin{equation}\label{defellsnrm}
\begin{aligned}
\|\x\|_s=\bigl(\sum_{i=1}^n |x_i|^s\bigr)^{1/s}, \, s\in[1,\infty],\,
\x=(x_1,\ldots,x_n)\in\R^n,\\
\|\x\|\defeq \|\x\|_2,\\
\rB(\x,r)=\{\y\in\R^n, \|\y-\x\|\le r\} \textrm{ for } r\ge 0.
\end{aligned}
\end{equation}

Let $f\in\rC^3(\rB(\x,r))$ for $r>0$.    Denote 
\begin{equation*}
f_{,i_1\ldots i_d}(\x)=\frac{\partial ^d}{\partial x_{i_1}\ldots\partial x_{i_d}}f(\x), \quad i_1,\ldots,i_d\in[n], d\in[3].
\end{equation*}
Recall the Taylor expansion of $f$ at $\x$ of order $3$ for $\bu\in\R^n$ with a small norm:
\begin{equation*}
\begin{aligned}
f(\x+\bu)\approx f(\x)+\nabla f(\x)^\top\bu+\frac{1}{2}\bu^\top \partial^2 f(\x)\bu+\frac{1}{6}\partial ^3 f(\x)\otimes \bu^{3\otimes},\\
\nabla f(\x)=(f_{,1}(\x),\ldots,f_{,n}(\x))^\top,  \quad \partial^2 f(\x)=[f_{,ij}(\x)],  i,j\in[n],\\ \partial^3 f(\x)=[f_{,ijk}(\x)], i,j,k\in[n], \, \partial^3 f(\x)\otimes\bu^{3\otimes}=\sum_{i,j,k\in[n]}f_{,ijk}(\x)u_i u_j u_k, 
\end{aligned}
\end{equation*}
where $\nabla f, \partial^2 f, \partial ^3 f$ are called
the gradient, the Hessian, and the 3-mode symmetric partial derivative tensor of $f$. 
A set $\rD\subset \R^n$ is called a domain if $\rD$ is an open connected set.
\begin{definition}\label{defconcconst}  
Assume that $f: \rD\to\R$ is a convex function 
 in a convex domain $\rD\subset \R^n$,  and $f\in\rC^3(\rD)$.   The function $f$ is called   $a (>0)$-self-concordant, or simply self-concordant, if the following inequality hold
\begin{equation}\label{defconcconst1} 
|\langle \partial^3f(\x),\otimes^3\bu\rangle|\le 2a^{-1/2} (\bu^\top \partial^2f(\bx)\bu)^{3/2}, 
\textrm{ for all }\x\in \rD,\bu\in\R^n.
\end{equation}
The function $f$ is called a standard self-concordant if $a=1$,  and a strongly $a$-self-concordant if $f(\x_m)\to\infty$ if the sequence $\{\x_m\}$ converges to the boundary of $\rD$.

The complexity value $\theta(f)\in[0,\infty]$ of an a-self-concordant function $f$ in $\rD$, called a self-concordant parameter in \cite[Definition 2.3.1]{NN94}, is
\begin{equation}\label{defsconc0}
\begin{aligned}
\theta(f)=sup_{\x\in\rD}
\inf\{\lambda^2\in[0,\infty], |\nabla f(\x)^\top \bu|^2\\
\le \lambda^2 a\big(\bu^\top\partial^2 f(\x)\bu\big),\forall \bu\in\R^n\}.
\end{aligned}
\end{equation}
 A strongly self-concordant function with a finite $\theta(f)$ is called a barrier (function).
\end{definition}

The following lemma is probably well known, and we give its short proof for completeness:
\begin{lemma}\label{scfstr}  Let $\rD\subset\R^n$ be a convex domain and assume that $f$ is a self-concordant function in $\rD$.  Then one of the following conditions hold
\begin{enumerate}[(a)]
\item The function $f$ is affine on $\rD$.
\item The Hessian $\partial^2 f$ is positive definite on $\rD$.
\item There is an orthogonal change of coordinates $\y=(y_1,\ldots,y_n)^\top=Q\x$, such that 
\begin{equation*}
f(\y)=f_1((y_1,\ldots,y_m)^\top)+by_{m+1}
\end{equation*} 
for some $m\in[n-1]$, such that $f_1$ has a positive definite Hessian on $\rD_1\subset \R^m$, where $\rD_1$ is the projection of $\rD$ on the first $m$ coordinates.
\end{enumerate}
Suppose furthermore that $\theta(f)<\infty$.   Then either $f$ is constant on $\rD$,  the Hessian of $f$ is positive definite in $\rD$, or the condition (c) holds with $b=0$.
In particular, $\nabla f(\x)$ is orthogonal to the kernel
of $\partial^2 f(\x)$ for $\x\in\rD$.
\end{lemma}
\begin{proof}
Corollary 2.1.1 in \cite{NN94} states that the nullity subspace $\U(\x)\subseteq\R^n$ of $\partial^2 f(\x),\x\in\rD$ does not depend on $\x\in\rD$.  Set $\U=\U(\x)$ for $\x\in\rD$.  If $\U=\R^n$ then (a) holds.   If $\U=\{\0\}$ then (b) holds.  If $\dim  \U=n-m$ it is straightforward to show that (c) holds. 

Assume that $\theta(f)<\infty$.   Suppose that $\0\ne\bu\in \ker \partial^2 f(\x)$.  Then $\nabla f(\x)^\top \bu=0$.   If the condition (a) satisfied then $f$ is a constant function.  Suppose that the condition (c) is satisfied.  Then $b=0$.
In particular, $\nabla f(\x)$ is orthogonal to the kernel
of $\partial^2 f(\x)$ for $\x\in\rD$.
\end{proof}

Denote by $\rS_n\supset\rS_{n,+}\supset \rS_{n,++}$ the space of $n\times n$ symmetric, the cone of positive semidefinite and the open set of positive definite matrices respecctively.   For $A,B\in \rS_n$ we denote $A\succ B\,(A\succeq B)$ if $A-B\in\rS_{n,++},\, (A-B\in\rS_{n,+})$.
For a matrix $A\in\R^{n\times n}$ denote by  $A^\dagger\in\R^{n\times n}$ the Moore-Penrose inverse of $A$ \cite[\textsection 4.12]{Frib}. Recall that if $A$ is invertible then $A^\dagger=A^{-1}$.  In particular for $a\in\R$: $a^\dagger =a^{-1}$ if $a\ne 0$, and $a^\dagger =0$ if $a=0$.
Assume that $A\in \rS_n$.  Then $A=Q\diag(\lambda_1,\ldots,\lambda_n) Q^\top$, where $Q\in\R^{n\times n}$ is an orthogonal matrix and
\begin{equation*}
\lambda_{\max}=\lambda_1\ge\ldots\ge\lambda_n=\lambda_{\min}
\end{equation*} 
are the eigenvalues of $A$. Then $A^\dagger=Q \diag(\lambda_1^\dagger,\ldots,\lambda_n^\dagger) Q^\top$.   In particular, $\ker A=\ker A^\dagger$.  In what follows we will use the following lemma:
\begin{lemma}\label{maxcharlem}
Let $A\in\rS_{n,+}, \y\in\R^{n}$.  Suppose that $\y^\top \ker A=0$.  Then
\begin{equation}\label{maxcharlem1}
\begin{aligned}
\y^\top A^\dagger\y=\inf\{\lambda>0, |\y^\top\bu|^2\le \lambda^2 \bu^\top A\bu, \\
\forall \bu\in\R^n\}= \max_{\bu\in\R^n} 2\y^\top \bu -\bu^\top A\bu.
\end{aligned}
\end{equation}
Furthermore,  if $A\succeq \y\y^\top$ then 
$\y^\top A^\dagger\y \le 1$.
\end{lemma}
\begin{proof} Clearly, it is enough to consider the case $\y\ne \0$. 
Let 
$$\mu=\inf\{\lambda>0, |\y^\top\bu|^2\le \lambda^2 \bu^\top A\bu, \forall \bu\in\R^n\}.$$
Suppose first that $A$ is positive definite.  Then  $\mu=\max_{\bu\ne \0}\big(\bu^\top\by \by^\top \bu\big)/\big(\bu^\top A\bu\big)$.  Let $B=\sqrt A\succ 0$, and $ B\bu=\bv$.   Then $\mu$ is the maximum eigenvalue of the rank-one matrix $B^{-1} \y\y^\top B^{-1}$.
Thus 
$$\mu=\tr B^{-1} \y\y^\top B^{-1}=\y^\top B^{2}\y=\y^\top A^{-1}\y.$$ 
This proves the first equality in \eqref{maxcharlem}.  

We now show the second equality in \eqref{maxcharlem}.  
Fix $\bw\ne \0$ and let $\bu=t\bw$.  Set  $\phi(t)=2\y^\top(t\bw)-(t\bw)^\top  A(t\bw)$.  The  maximum of $\phi(t)$ is achieved at $t=\frac{\y^\top\bw}{\bw^\top A\bw}$ and is equal to $\frac{|\y^\top \bw|^2}{\bw^\top A\bw}$.   Use the first equality of \eqref{maxcharlem} to deduce the second equality in \eqref{maxcharlem1}.

Assume now that $A\in\rS_{n,+}$ is singular.
Suppose that $\bu^\top A\bu=0$.  Then $\bu\in\ker A$.  Hence, $\y^\top \bu=0$.  Therefore, its is enough to consider the case where $\bu\in$range$\, A$.   As $\y^\top\ker A=0$ it follows that $\y\in$range$\,A$.  Let $C$ be the restriction of $A$, viewed as a linear operator $\R^n\to\R^n$,  to range~$A$.  So $C$ is positive definite and we can use the previous case.
Observe that $\y^\top C^{-1}\y=\y^\top A^\dagger\y$, and the first equality of  \eqref{maxcharlem} follows.  The second equality follows similarly.  

Suppose that $A\succeq \y\y^\top$.  That is, $A=\y\y^\top +B$ for some $B\in\rS_{n,+}$.
Hence $\ker A=\ker (\y\y^\top)\cap\ker B$.  Therefore $\y$ is orthogonal to $\ker A$.
We now use the second equality in \eqref{maxcharlem1}.  Observe that $\bu^\top A\bu\ge \bu^\top (\y\y^\top)\bu$.   Hence
\begin{equation*}
\y^\top A^\dagger \y\le \y^\top (\y\y^\top)^\dagger\y=1.
\end{equation*}
\end{proof}

\begin{corollary}\label{charthetf}  Let $\rD\subset\R^n$ be a convex domain and assume that $f$ is a nonconstant $a$-self-concordant function in $\rD$.    Suppose furthermore that $\nabla f(\x)^\top \ker \partial^2 f(\x)=0$ for each $\x\in\rD$.   Then
\begin{equation}\label{defconcconst1}
 \theta(f)=a^{-1}\sup_{\x\in \rD}\nabla f(\x)^\top (\partial^2f)^\dagger(\x)\nabla f(\x).
\end{equation}
\end{corollary}
This equality is well known if $\partial^2f$ is invertible \cite[top of page 16]{NN94}.
A simple example of a strongly standard self-concordant function for the interior of the cone $\R_+^n$,  denoted as $\R_{++}^n$, where $\R_{++}=(0,\infty)$, is 
\begin{equation}\label{logxbar}
\begin{aligned}
\sigma(\x)=-\sum_{i=1}^n \log x_i,\\
\theta(\sigma)=n.
\end{aligned}
\end{equation}
The equality $a=1$ follows from the well known fact that the norm $\|\x\|_s$ is decreasing for $s\in[1,\infty]$.  Use Corollary \ref{charthetf} 
to deduce the second equality of \eqref{logxbar}.

 Assume that $\rD$ is a bounded convex domain,  $\x\in\rD$  and $\rL$ is a line  through $\x$. 
 Denote by $d_{\max}(\x,\rL)\ge d_{min}(\x,\rL)$ the two distances from $\x$ to the end points of $\rL\cap\partial\rD$.   Then $\textrm{sym}(\x,\rD)$ is the infimum of $\frac{d_{\min}(\x,L)}{ d_{\max}(\x,L)}$ for all lines $\rL$ through $\x$.   Observe that if $\rB(\x,r)\subset \textrm{Closure}(\rD)\subset \rB(\x,R)$ then $\textrm{sym}(\x,\rD)\ge r/R$.

Recall that Renegar \cite{Ren01} deals only with strongly standard self-concordant functions.  
The complexity value $\theta(f)$, coined in \cite{Ren01}, is called the parameter of barrier $f$ in \cite{NN94}, and is considered only for self-concordant barrier  in \cite[\textsection 2.3.1] {NN94}.   

We now recall the  complexity result to approximate the infimum of a linear functional on a bounded convex domain with whose boundary is given by a barrier function $\beta$.
We normalize $\beta$ by assuming that it is strongly self-concordant.
A simple implementation of the Newton's method  is  a "short-step" ipm's that  follows the central path \cite[\textsection 2.4.2]{Ren01}.
The number of iterations to approximate the minimum of a linear functional  within $\varepsilon$ precision starting with an intial point $\x'$ is \cite[Theorem 2.4.1]{Ren01}:
\begin{equation}\label{Renthm}
O\big(\sqrt{\theta(\beta)}\log\big(\frac{\theta(\beta)}{\varepsilon \textrm{sym}(\x',\rD)}\big)\big).
\end{equation}
\subsection{The number of iterations of ipm for matrix optimal transport}\label{subsec:mot}
Assume that $1<m,n\in\N$.
Let $\bp=(p_1,\ldots,p_m)^\top,\bq=(q_1,\ldots,q_n)^\top$ be two positive probability vectors.  Denote by $\rU(P)$ the set \eqref{defUPmat}, where $n_1=m, n_2=n$ and $\bp_1=\bp, \bp_2=\bq$.   Set $\1_m=(1,\ldots,1)^\top\in\R^m$, and define
\begin{equation*}
\rU_{0,2}=\{X=[x_{ij}]\in\R^{m\times n},  X\1_n=\0, X^\top \1_m=\0\}.
\end{equation*}
As the sum of all rows of $X$ is equal to the sum of all columns of $X$ it follows that $\dim\rU_{0,2}=(m-1)(n-1)$.  
Let $\be_i=(\delta_{1i},\ldots,\delta_{mi})^\top\in\R^m, i\in[m]$ and $\bbf_j=(\delta_{1j},\ldots,\delta_{nj})^\top\in\R^n, j\in[n]$ be the standard bases in $\R^m$ and $\R^n$ respectively.  Then one has a following simple basis in $\rU_{0,2}$:
\begin{equation*}
\begin{aligned}
\bg_i\bh_j^\top, i\in[m-1], j\in[n-1],\\
\bg_i=\be_i-\be_{i+1}, i\in[m-1],\quad \bh_j=\bbf_j-\bbf_{j+1},j\in[n-1].
\end{aligned}
\end{equation*}

The interior of $\rU(P)$, denoted as $\rU_o(P)$,  is the set of positive matricers in the affine space 
$$\rU_{af,2}(P)=\{X=\bp\bq^\top +\sum_{i=1}^{m-1}\sum_{j=1}^{n-1} t_{ij}\bg_i\bh_j^\top, \quad t_{ij}\in\R,  i\in[m-1],j\in[n-1]\}. $$
Let  
$$\sigma(X)=-\sum_{i=1}^m\sum_{j=1}^n \log x_{ij}, \quad X=[x_{ij}]\in\R_{++}^{m\times n}$$
be a barrier function $\R_{++}^{m\times n}$.  Recall that $\sigma$ is a standard self-concondant barrier with $\theta(\sigma)=mn$.   The restriction of $\tilde \sigma$ to $\rU_{af,2}\cap \R_{++}^{m\times n}$ is a standard self-concordant barrier with 
\begin{equation}\label{thetsig}
\theta(\tilde\sigma)\le mn.
\end{equation}
\begin{theorem}\label{ipmotmat}
Let $\bp=(p_1,\ldots,p_m)^\top\in\R^m,\bq=(q_1,\ldots,q_n)^\top\in\R^n$ be positive probability vectors.  Consider the minimum problem \eqref{TOTmat} on the polytope
 $\rU(P)$ given by \eqref{defUPmat}, where $n_1=m, n_2=n$ and $\bp_1=\bp, \bp_2=\bq$. 
The short step interior path algorithm with the barrier $\tilde \sigma$ starting at the point $\bp\bq^\top$ finds the value $\tau(C,P)$ within precision $\varepsilon>0$ in  
\begin{equation}\label{ipmotmat1}
O\big(\sqrt{mn}\log\frac{\sqrt{2}mn}{\varepsilon (\min_{i\in[m]}p_i)(\min_{j\in[n]}q_j)}\big)
\end{equation}
iterations.
\end{theorem}
\begin{proof} In view of \eqref{Renthm} and \eqref{thetsig}  it is enough to show that 
$$\textrm{sym}(\bp\bq^\top,\rD)\ge (\min_{i\in[m]}p_i)(\min_{j\in[n]}q_j)\big)/\sqrt{2},\quad
\rD=\rU_{af}(P)\cap \R_{++}^{m\times n}.$$ 
For $Y=[y_{ij}]\in \R^{m\times n}$ denote: $\|Y\|_2=\sqrt{\sum_{i=1}^m\sum_{j=1}^n y_{ij}^2}$.
Assume that $X\in \textrm{Closure}(\rD)$.   Then  $X=[x_{ij}]$ is a  probability matrix.  
Note that if $X\in\partial \rD$ then $x_{ij}=0$ for some $i\in[m],j\in[n]$. Hence,  for $X\in\partial \rD$ we have the inequalities:
\begin{equation*}
\begin{aligned}
\sqrt{2}\ge\|\bp\bq^\top-X\|_2\ge p_iq_j   \textrm{ for some } i\in[m],j\in[n]\Rightarrow\\ 
 \sqrt{2}\ge d_{\max}(\bp\bq^\top,L)\ge d_{\min}(\bp\bq^\top,L)\ge(\min_{i\in[m]}p_i)(\min_{j\in[n]}q_j)\big).
\end{aligned}
\end{equation*}
\end{proof}
\subsection{The number iterations of ipm for tensor optimal transport}\label{subsec:totipm}
Assume that $d>2$.
We first consider the TOT of the form \eqref{TOT}.  We now repeat the arguments of the previous subsection.  Let 
\begin{equation*}
\begin{aligned}
\rU_{0,d}=\{\cU\in \otimes^d\R^n, \cU\times_{\bar k} \cJ_{d-1}=0, k\in[d]\},,\\
\rU_{af,d}(P)=\otimes_{k=1}^d \bp_k +\rU_{0,d}, \quad P=(\bp_1,\ldots,\bp_d)
\end{aligned}
\end{equation*}
Then $\rU(P)=\rU_{af,d}(P)\cap \otimes_{k=1}^d \R_+^{n_k}$.  The interior of $\rU(P)$ is given by $\rU_{af,d}\cap \otimes_{k=1}^d \R_{++}^{n_k}$.  Let
$$\sigma(\cX)=-\sum_{i_k\in[n_k], k\in[d]}\ \log x_{i_1,\ldots,i_d}, \quad \cX\in\otimes_{k=1}^d\R_{++}^{n_k}$$
be a barrier function on $\otimes_{k=1}^d\R_{++}^{n_k}$.   Thus, $\sigma$ is a standard self-concondant barrier with $\theta(\sigma)=\prod_{k=1}^d n_k$.   The restriction of $\tilde \sigma$ to $\rU_{af,2}\cap \R_{++}^{m\times n}$ is a standard self-concordant barrier with 
\begin{equation}\label{thetsigd}
\theta(\tilde\sigma)\le \prod_{k=1}^d n_k.
\end{equation}
The arguments of the proof of Theorem \ref{ipmotmat} yield:
\begin{theorem}\label{ipmotd}
Let $\bp_k=(p_{1,k},\ldots,p_{n_k,k})^\top\in\R^{n_k}, k\in[d]$ be positive probability vectors.  Consider the minimum problem \eqref{TOT} on the polytope
 $\rU(P)$ given by \eqref{defU(P)ten}. 
The short step interior path algorithm with the barrier $\tilde \sigma$ starting at the point $\otimes_{k=1}^d\bp_k$ finds the value $\tau(C,P)$ within precision $\varepsilon>0$ in  
\begin{equation}\label{ipmotmat1}
O\big(\sqrt{\prod_{k=1}^d n_k}\log\frac{\sqrt{2}\prod_{k=1}^d n_k}{\varepsilon \prod_{k=1}^d\min_{i_k\in[n_k]}p_{i_k,k}}\big)
\end{equation}
iterations.
\end{theorem}

We now consider a variation of the polytope $\rU(P)$, which correspond to the problem of $d$-dimensional stochastic tensors \cite{LZZ17}.   For $\cU=[u_{i_1,\ldots,i_d}]\in \otimes_{k=1}^d \R^{n_k}$ and a vector $\x=(x_1,\ldots,x_{n_k})^\top\in\R^{n_k}$ denote 
\begin{equation*}
\begin{aligned}
\cU\times _k \x =\cW=[w_{i_1,\ldots,i_{k-1},i_{k+1},\ldots,i_d} ]\in \otimes_{j\in[d]\setminus\{k\}}\R^{n_j},\\
w_{i_1,\ldots,i_{k-1},i_{k+1},\ldots,i_d}=\sum_{i_k=1}^{n_k} u_{i_1,\ldots,i_d} x_{i_k}.
\end{aligned}
\end{equation*}
Define
\begin{equation}\label{defV(P)ten}
\begin{aligned}
\rV(P)=\{\cV=[v_{i_1,\ldots,i_d}]\in \otimes_{k=1}^d\R_+^{n_k}, \\
\cV\times_{k} \1_{n_k}=\otimes_{j\in[d]\setminus{k}}\bp_j, k\in[d]\},  P=(\bp_1,\ldots,\bp_d),\\
\end{aligned}
\end{equation}
Let 
\begin{equation}\label{defV0(P)}
\begin{aligned}
\rV_{0,d}=\{\cV=[v_{i_1,\ldots,i_d}]\in \otimes_{k=1}^d\R^{n_k}, 
\cV\times_{k} \1_{n_k}=0\in[d]\},\\
\rV_{af,d}(P)=\otimes_{k=1}^d \bp_k+\rV_{0,d}.
\end{aligned}
\end{equation}
Let $\be_{1,k},\ldots,\be_{n_k,k}\in\R^{n_k}$ be the standard basis in $\R^{n_k}$ for $k\in[d]$.   Denote
\begin{equation*}
\bg_{i,k}=\be_{i,k}-\be_{i+1,k} \textrm{ for } i\in[n_k-1].
\end{equation*}
Observe that $\1_{n_k}^\perp=$span$(\bg_{1,k},\ldots,\bg_{n_k-1,k})$ is the orthogonal
complement of $\1_{n_k}$ in $\R^{n_k}$.
We claim that 
\begin{equation}\label{V0dstr}
\rV_{0,d}=\otimes_{k=1}^d\1_{n_k}^\perp.
\end{equation}
Indeed, assume that $\cV\in\otimes_{k=1}^d \R^{n_k}$ satisfies $\cV\times_1\1_{n_1}=0$.  View $\cV$ as a matrix in $\R^{n_1}\otimes\big(\otimes_{k=2}^d\R^{n_k}\big)$.  The above condition  yields that range~$ \cV\subset \1_{n_1}^\perp$, which is equivalent to $\cV\in \1_{n_1}^\perp\otimes\big(\otimes_{k=2}^d \R^{n_k}\big)$.  Apply this observation to $\rV_{0,d}$ to deduce \eqref{V0dstr}.  Hence,
\begin{equation*}
\rV_{af,d}(P)=\{\cV=\otimes_{k=1}^d \bp_k +\sum_{i_k\in[n_k-1],k\in[d]} t_{i_1,\ldots,i_d}\otimes_{j=1}^d \bg_{i_j,j}, [t_{i_1,\ldots,i_d}]\in\otimes_{l=1}^d \R^{n_l-1}\}. 
\end{equation*}
Then $\rV(P)=\rV_{af,d}(P)\cap \otimes_{k=1}^d \R_+^{n_k}$.  The interior of $\rV(P)$ is given by $\rV_{af,d}\cap \otimes_{k=1}^d \R_{++}^{n_k}$.  

The arguments of the proof of Theorem \ref{ipmotmat} yield:
\begin{theorem}\label{ipmotdv}
Let $\bp_k=(p_{1,k},\ldots,p_{n_k,k})^\top\in\R^{n_k}, k\in[d]$ be positive probability vectors.  Consider the minimum problem $
\tau(\cC,P)= \min\{\langle \cC, \cU\rangle, \cU\in \rV(P)\}$.
The short step interior path algorithm with the barrier $\tilde \sigma$ starting at the point $\otimes_{k=1}^d\bp_k$ finds the value $\tau(C,P)$ within precision $\varepsilon>0$ in the number of iterations given by \eqref{ipmotmat1}.
\end{theorem}
\subsection{Iteration estimates for certain probabilities $P$}\label{subsec:cest}
Our iteration estimate \eqref{ipmotmat1} depends on $P=(\bp_1,\ldots,\bp_d)$: the product of the minum values of the coordinates of $\bp_k$ for $k\in[d]$.   
\begin{definition}\label{defwKldist}
A probability vector $\bp=(p_1,\ldots,p_n)\in\R_{+}^n$ is called a weak $K-\ell$ uniform distribution if
\begin{equation}\label{defwKldist1}
p_j\ge \frac{K}{n^{\ell}}, j\in[n],\quad K>0, \ell\ge 1, K\le n^{\ell-1}.
\end{equation}
\end{definition}
Note that if $K=\ell=1$ the $\bp$ is the uniform distribution.

Theorem \ref{ipmotd} yields:
\begin{corollary}\label{cipmotd}  Let the assumptions Theorem \ref{ipmotd} hold.
Assume that each $\bp_j$ is a weak $K-\ell$ uniform distribution.  Then the short step interior path algorithm with the barrier $\tilde \sigma$ starting at the point $\otimes_{k=1}^d\bp_k$ finds the value $\tau(C,P)$ within precision $\varepsilon>0$ in  
\begin{equation}\label{cipmotmat1}
O\big(\sqrt{\prod_{k=1}^d n_k}\log (\sqrt{2}\varepsilon^{-1}K^{-d}\prod_{k=1}^d n_k^{1+\ell})\big)
\end{equation}
iterations. In particular, if $n_1=\cdots=n_d=n$ then the above estimate is $O\big(n^{d/2}\log(\sqrt{2}\varepsilon^{-1} K^{-d} n^{d(1+\ell)})\big)$. 
\end{corollary}
 \section*{Acknowledgment}
The author is partially supported by the Simons Collaboration Grant for Mathematicians.  
\bibliographystyle{plain}

\end{document}